\documentclass[12pt]{amsart}
\usepackage{graphicx}
\usepackage{amsmath}
\usepackage{amssymb}
\usepackage{setspace}
\newtheorem{thm}{Theorem}[section]

\newtheorem{lem}[thm]{Lemma}

\theoremstyle{definition}

\theoremstyle{remark}
\newtheorem{rem}[thm]{Remark}
\theoremstyle{conclusion}

\theoremstyle{question}
\numberwithin{equation}{section}

\begin{document}
\title[Critical order Lane-Emden-Hardy equations in $R^n$]{Liouville type theorem for critical order Lane-Emden-Hardy equations in $\mathbb{R}^n$}

\author{Wenxiong Chen, Wei Dai, Guolin Qin}

\address{Department of Mathematics, Yeshiva University, New York, NY, USA}
\email{wchen@yu.edu}

\address{School of Mathematics and Systems Science, Beihang University (BUAA), Beijing 100083, P. R. China}
\email{weidai@buaa.edu.cn}

\address{School of Mathematics and Systems Science, Beihang University (BUAA), Beijing 100083, P. R. China}
\email{qinbuaa@foxmail.com}

\thanks{The first author is partially supported by the Simons Foundation Collaboration Grant for Mathematicians 245486. The second author is supported by the NNSF of China (No. 11501021).}

\begin{abstract}
In this paper, we are concerned with the critical order Lane-Emden-Hardy equations
\begin{equation*}
  (-\Delta)^{\frac{n}{2}}u(x)=\frac{u^{p}(x)}{|x|^{a}} \,\,\,\,\,\,\,\,\,\,\,\, \text{in} \,\,\, \mathbb{R}^{n}
\end{equation*}
with $n\geq4$ is even, $0\leq a<n$ and $1<p<+\infty$. We prove Liouville theorem for nonnegative classical solutions to the above Lane-Emden-Hardy equations (Theorem \ref{Thm0}), that is, the unique nonnegative solution is $u\equiv0$. Our result seems to be the first Liouville theorem on the critical order equations in higher dimensions ($n\geq3$).
\end{abstract}
\maketitle {\small {\bf Keywords:} Critical order; Lane-Emden-Hardy equations; Liouville theorem; Nonnegative solutions; Super poly-harmonic properties. \\

{\bf 2010 MSC} Primary: 35B53; Secondary: 35J61, 35J91.}

\section{Introduction}

In this paper, we investigate the Liouville property of nonnegative solutions to the following critical order Lane-Emden-Hardy equations
\begin{equation}\label{PDE}\\\begin{cases}
(-\Delta)^{\frac{n}{2}}u(x)=\frac{u^{p}(x)}{|x|^{a}} \,\,\,\,\,\,\,\,\,\, \text{in} \,\,\, \mathbb{R}^{n}, \\
u(x)\geq0, \,\,\,\,\,\,\,\, x\in\mathbb{R}^{n},
\end{cases}\end{equation}
where $u\in C^{n}(\mathbb{R}^{n})$ if $a=0$, $u\in C^{n}(\mathbb{R}^{n}\setminus\{0\})\cap C^{n-2}(\mathbb{R}^{n})$ if $0<a<n$, $n\geq4$ is even and $1<p<+\infty$.

For $0<\alpha\leq n$, PDEs of the form
\begin{equation}\label{GPDE}
  (-\Delta)^{\frac{\alpha}{2}}u(x)=|x|^{a}u^{p}(x)
\end{equation}
are called the fractional order or higher order Hardy (Lane-Emden, H\'{e}non) equations for $a<0$ ($a=0$, $a>0$, respectively), which have many important applications in conformal geometry and Sobolev inequalities. We say equations \eqref{GPDE} have critical order if $\alpha=n$ and non-critical order if $0<\alpha<n$. Being essentially different from the non-critical order equations, the fundamental solution $c_{n}\ln\frac{1}{|x-y|}$ of $(-\Delta)^{\frac{n}{2}}$ changes its signs in critical order case $\alpha=n$, thus the integral representation in terms of the fundamental solution can't be deduced directly from the super poly-harmonic properties. Liouville type theorems for equations \eqref{GPDE} (i.e., nonexistence of nontrivial nonnegative solutions) have been quite extensively studied (see \cite{BG,CD,CFY,CL,CL1,DQ1,DQ,GS,Lin,MP,P,PS,WX} and the references therein). It is crucial in establishing a priori estimates and existence of positive solutions for non-variational Dirichlet problems of a class of elliptic equations (see \cite{GS1,PQS}).

In the special case $a=0$, equation \eqref{GPDE} becomes the well-known Lane-Emden equation, which also arises as a model in astrophysics. For $\alpha=2$ and $1<p<p_{s}:=\frac{n+2}{n-2}$ ($:=\infty$ if $n=2$), Liouville type theorem was established by Gidas and Spruck in their celebrated article \cite{GS}. Later, the proof was simplified to a large extent by Chen and Li in \cite{CL} using the Kelvin transform and the method of moving planes (see also \cite{CL1}). For $n>\alpha=4$ and $1<p<\frac{n+4}{n-4}$, Lin \cite{Lin} proved the Liouville type theorem for all the nonnegative $C^{4}(\mathbb{R}^{n})$ smooth solutions of \eqref{GPDE}. When $\alpha\in(0,n)$ is an even integer and $1<p<\frac{n+\alpha}{n-\alpha}$, Wei and Xu established Liouville type theorem for all the nonnegative $C^{\alpha}(\mathbb{R}^{n})$ smooth solutions of \eqref{GPDE} in \cite{WX}. For general $a\neq0$, $0<\alpha<n$, $0<p<\frac{n+\alpha+2a}{n-\alpha}$ ($1<p<+\infty$ if $\alpha=n=2$), there are also lots of literatures on Liouville type theorems for general fractional order or higher order Hardy-H\'{e}non equations \eqref{GPDE}, for instance, Bidaut-V\'{e}ron and Giacomini \cite{BG}, Chen and Fang \cite{CF}, Dai and Qin \cite{DQ1}, Gidas and Spruck \cite{GS}, Mitidieri and Pohozaev \cite{MP}, Phan \cite{P}, Phan and Souplet \cite{PS} and many others. For Liouville type theorems on systems of PDEs of type \eqref{GPDE} with respect to various types of solutions (e.g., stable, radial, nonnegative, sign-changing, $\cdots$), please refer to \cite{BG,DQ,FG,M,P,PQS,S,SZ} and the references therein.

For the critical nonlinearity cases $p=\frac{n+\alpha}{n-\alpha}$ with $a=0$ and $0<\alpha<n$, the quantitative and qualitative properties of solutions to fractional order or higher order equations \eqref{GPDE} have also been widely studied. In the special case $n>\alpha=2$, equation \eqref{GPDE} becomes the well-known Yamabe problem (for related results, please see Gidas, Ni and Nirenberg \cite{GNN1,GNN}, Caffarelli, Gidas and Spruck \cite{CGS} and the references therein). For $n>\alpha=4$, Lin \cite{Lin} classified all the positive $C^{4}$ smooth solutions of \eqref{GPDE}. In \cite{WX}, among other things, Wei and Xu proved the classification results for all the positive $C^{\alpha}$ smooth solutions of \eqref{GPDE} when $\alpha\in(0,n)$ is an even integer. For $n>\alpha=3$, Dai and Qin \cite{DQ1} classified the positive $C^{3,\epsilon}_{loc}\cap\mathcal{L}_{1}$ classical solutions of \eqref{GPDE}. In \cite{CLO}, by developing the method of moving planes in integral forms, Chen, Li and Ou classified all the positive $L^{\frac{2n}{n-\alpha}}_{loc}$ solutions to the equivalent integral equation of the PDE \eqref{GPDE} for general $\alpha\in(0,n)$, as a consequence, they obtained the classification results for positive weak solutions to PDE \eqref{GPDE}. Subsequently, Chen, Li and Li \cite{CLL} developed a direct method of moving planes for fractional Laplacians $(-\Delta)^{\frac{\alpha}{2}}$ with $0<\alpha<2$ and classified all the $C^{1,1}_{loc}\cap\mathcal{L}_{\alpha}$ positive solutions to the PDE \eqref{GPDE} directly as an application, where the function space
\begin{equation}\label{0-0}
  \mathcal{L}_{\alpha}(\mathbb{R}^{n}):=\Big\{f: \mathbb{R}^{n}\rightarrow\mathbb{R}\,\big|\,\int_{\mathbb{R}^{n}}\frac{|f(x)|}{1+|x|^{n+\alpha}}dx<\infty\Big\}.
\end{equation}
In the limiting (i.e., critical order) case $n=\alpha=2$, Chen and Li \cite{CL1} classified all the $C^{2}$ smooth solutions with finite total curvature of the equation
\begin{equation}\label{0-1}\\\begin{cases}
-\Delta u=e^{2u}, \,\,\,\,\,\,\,\, x\in\mathbb{R}^{2}, \\
\int_{\mathbb{R}^{2}}e^{2u}dx<\infty.
\end{cases}\end{equation}
In general, when $\alpha=n$, under some assumptions, Chang and Yang \cite{CY} classified the smooth solutions to the critical order equations
\begin{equation}\label{0-3}
  (-\Delta)^{\frac{n}{2}}u=(n-1)!e^{nu}.
\end{equation}
When $n=\alpha=4$, Lin \cite{Lin} proved the classification results for all the $C^{4}$ smooth solutions of
\begin{equation}\label{0-2}\\\begin{cases}
\Delta^{2}u=6e^{4u}, \,\,\,\,\,\,\,\, x\in\mathbb{R}^{4}, \\
\int_{\mathbb{R}^{4}}e^{4u}dx<\infty, \,\,\,\,\,\, u(x)=o(|x|^{2}) \,\,\,\, \text{as} \,\,\,\, |x|\rightarrow\infty.
\end{cases}\end{equation}
When $\alpha=n$ is an even integer, Wei and Xu \cite{WX} classified all the $C^{n}$ smooth solutions of \eqref{0-3} with finite total curvature (i.e., $\int_{\mathbb{R}^{n}}e^{nu}dx<\infty$) under the assumption $u(x)=o(|x|^{2})$ as $|x|\rightarrow\infty$. Recently, under the same assumption $u(x)=o(|x|^{2})$ as $|x|\rightarrow\infty$, Chen and Zhang \cite{CZ} classified all the smooth solutions of the critical order equations \eqref{0-3} with finite total curvature for arbitrary space dimensions $n$ (no matter $n$ is even or odd) via a unified approach. In particular, one should note that, when $n$ is odd, \eqref{0-3} is a fractional equation of nonlocal nature. For more literatures on the quantitative and qualitative properties of solutions to fractional order or higher order conformally invariant PDE and IE problems, please refer to \cite{CD,CL1,CZ,DFHQW,DFQ,Zhu} and the references therein.

One should observe that, all the literatures on Liouville type theorems for PDE \eqref{GPDE} mentioned above are focused on the non-critical order and subcritical nonlinearity cases $0<\alpha<n$ and $1<p<\frac{n+\alpha+2a}{n-\alpha}$ except $n=2$. In this paper, we will establish Liouville type theorem for nonnegative classical solutions of \eqref{GPDE} in critical order cases, that is, $\alpha=n\geq4$ is even and $1<p<+\infty$. Our theorem seems to be the first result on this problem.

Our Liouville type result for \eqref{PDE} is the following theorem.
\begin{thm}\label{Thm0}
Assume $n\geq4$ is even, $0\leq a<n$, $1<p<+\infty$ and $u$ is a nonnegative solution of \eqref{PDE}. If one of the following two assumptions
\begin{equation*}
  0\leq a\leq2+2p \,\,\,\,\,\,\,\,\,\,\,\, \text{or} \,\,\,\,\,\,\,\,\,\,\,\, u(x)=o(|x|^{2}) \,\,\,\, \text{as} \,\, |x|\rightarrow+\infty
\end{equation*}
holds, then $u\equiv0$ in $\mathbb{R}^{n}$.
\end{thm}

\begin{rem}\label{remark0}
In Theorem \ref{Thm0}, the smoothness assumption on $u$ at $x=0$ is necessary. Equation \eqref{PDE} admits a distributional solution of the form $u(x)=C|x|^{-\sigma}$ with $\sigma=\frac{n-a}{p-1}>0$.
\end{rem}

It's well known that the super poly-harmonic properties of solutions are crucial in establishing Liouville type theorems and the representation formulae for higher order or fractional order PDEs (see e.g. \cite{CF,CL2,WX}). In Section 2, we will first prove the super poly-harmonic properties of solutions by using ``re-centers and iteration" arguments (see Lemma \ref{lemma0}). Nevertheless, being different from the non-critical order equations, the integral representation in terms of the fundamental solution of $(-\Delta)^{\frac{n}{2}}$ can't be deduced directly from the super poly-harmonic properties, since the fundamental solution $c_{n}\ln\frac{1}{|x-y|}$ changes its signs in $\mathbb{R}^{n}$. Fortunately, based on Lemma \ref{lemma0}, we can derive instead the following integral inequality (see \eqref{formula})
\begin{equation*}
  +\infty>u(0)\geq\int_{\mathbb{R}^{n}}\frac{R_{2,n}}{|y^{\frac{n}{2}}|^{n-2}}\int_{\mathbb{R}^{n}}\frac{R_{2,n}}{|y^{\frac{n}{2}}-y^{\frac{n}{2}-1}|^{n-2}}\cdots
  \int_{\mathbb{R}^{n}}\frac{R_{2,n}}{|y^{2}-y^{1}|^{n-2}}\frac{u^{p}(y^{1})}{|y^{1}|^{a}}dy^{1}\cdots dy^{\frac{n}{2}}
\end{equation*}
for $0\leq a<2$, where the Riesz potential's constants $R_{2,n}:=\frac{\Gamma\big(\frac{n-2}{2}\big)}{4\pi^{\frac{n}{2}}}$. This integral inequality will lead to a contradiction on integrability unless $u\equiv0$. As to the cases $a\geq2$, we can also obtain a contradiction using the integral estimates arguments if $u$ is not identically zero. As a consequence, Theorem \ref{Thm0} is proved.

In the following, we will use $C$ to denote a general positive constant that may depend on $n$, $a$, $p$ and $u$, and whose value may differ from line to line.

\section{Proof of Theorem \ref{Thm0}}

In this section, we will prove Theorem \ref{Thm0} by using contradiction arguments. Now suppose on the contrary that $u\geq0$ satisfies equation \eqref{PDE} but $u$ is not identically zero, then there exists some $\bar{x}\in\mathbb{R}^{n}$ such that $u(\bar{x})>0$.

The super poly-harmonic properties of solutions are closely related to the representation formulae and Liouville type theorems (see \cite{CF,CL2,WX} and the references therein). Therefore, in order to prove Theorem \ref{Thm0}, we need the following lemma about the super poly-harmonicity.

\begin{lem}\label{lemma0}(Super poly-harmonic properties). Assume $n\geq4$ is even, $0\leq a<n$, $1<p<+\infty$ and $u$ is a nonnegative solution of \eqref{PDE}. If one of the following two assumptions
\begin{equation*}
  0\leq a\leq2+2p \,\,\,\,\,\,\,\,\,\,\,\, \text{or} \,\,\,\,\,\,\,\,\,\,\,\, u(x)=o(|x|^{2}) \,\,\,\, \text{as} \,\, |x|\rightarrow+\infty
\end{equation*}
holds, then
\begin{equation*}
  (-\Delta)^{i}u(x)\geq0
\end{equation*}
for every $i=1,2,\cdots,\frac{n}{2}-1$ and all $x\in\mathbb{R}^{n}$.
\end{lem}
\begin{proof}
Let $u_{i}:=(- \Delta)^{i}u$. We want to show that $u_{i}\geq0$ for $i=1,2,\cdots,\frac{n}{2}-1$. Our proof will be divided into two steps.

\textbf{\emph{Step 1.}} We first show that
\begin{equation}\label{2-1}
u_{\frac{n}{2}-1}=(-\Delta)^{\frac{n}{2}- 1}u\geq0.
\end{equation}
If not, then there exists $0\neq x^{1}\in\mathbb{R}^n$, such that
\begin{equation}\label{2-2}
  u_{\frac{n}{2}-1}(x^{1})<0.
\end{equation}

Now, let
\begin{equation}\label{2-3}
  \bar{f}(r)=\bar{f}\big(|x-x^1|\big):=\frac{1}{|\partial B_{r}(x^{1})|}\int_{\partial B_{r}(x^{1})}f(x)d\sigma
\end{equation}
be the spherical average of $f$ with respect to the center $x^1$. Then, by the well-known property $\overline{\Delta u}=\Delta\bar{u}$ and $0\leq a<n$, we have, for any $r\geq0$ and $r\neq|x^{1}|$,
\begin{equation}\label{2-4}
\left\{{\begin{array}{l} {-\Delta\overline{u_{\frac{n}{2}-1}}(r)=\overline{\frac{u^{p}(x)}{|x|^{a}}}(r)}, \\  {} \\ {-\Delta\overline{u_{\frac{n}{2}-2}}(r)=\overline{u_{\frac{n}{2}-1}}(r)}, \\ \cdots\cdots \\ {-\Delta\overline u(r)=\overline{u_1}(r)}. \\ \end{array}}\right.
\end{equation}
From the first equation in \eqref{2-4}, by Jensen's inequality, we get, for any $r\geq0$ and $r\neq|x^{1}|$,
\begin{align}\label{2-5}
-\Delta\overline{u_{\frac{n}{2}-1}}(r)&=\frac{1}{{| {\partial
B_{r}({x^{1}})}| }}\int_{\partial B_{r}(
{x^{1}})}\frac{{u^{p}(x)}}{|x|
^{a}}d\sigma\nonumber\\
& \geq({r+| {x^{1}}| })^{-a}\frac
{1}{{| {\partial B_{r}({x^{1}})}| }}
\int_{\partial B_{r}({x^{1}})}{u^{p}(  x)}d\sigma\\
&  \geq({r+| {x^{1}}| })^{-a}\left(
{\frac{1}{{| {\partial B_{r}({x^{1}})}| }
}\int_{\partial B_{r}({x^{1}})}{u(x)
}d\sigma}\right)^{p}\nonumber\\
& =({r+| {x^{1}}| })^{-a}\bar{u}^{p}(r)\geq0.\nonumber
\end{align}
From \eqref{2-5}, one has
\begin{equation}\label{2-6}
  -\frac{1}{r^{n-1}}\Big(r^{n-1}\overline{u_{\frac{n}{2}-1}}\,'(r)\Big)'\geq({r+| {x^{1}}| })^{-a}\bar{u}^{p}(r)\geq0.
\end{equation}
Since $0\leq a<n$, we can integrate both sides of \eqref{2-6} from $0$ to $r$ and derive
\begin{equation}\label{2-7}
\overline{u_{\frac{n}{2}-1}}\,'(r)\leq0, \,\,\,\,\,\, \overline{u_{\frac{n}{2}-1}}(r)\leq\overline{u_{\frac{n}{2}-1}}(0)=u_{\frac{n}{2}-1}(x^{1})=:-c_{0}<0
\end{equation}
for any $r\geq0$. From the second equation in \eqref{2-4}, we deduce that
\begin{equation}\label{2-8}
-\frac{1}{{r^{n-1}}}\Big({r^{n-1}\overline{u_{\frac{n}{2}-2}}\,'}(r)\Big)'=\overline{u_{\frac{n}{2}-1}}(r)\leq-c_{0}, \,\,\,\,\,\, \forall \,\, r\geq0,
\end{equation}
integrating from $0$ to $r$ yields
\begin{equation}\label{2-9}
\overline{u_{\frac{n}{2} - 2}}\,'(r)\geq\frac{c_{0}}{n}r, \,\,\,\,\,\, \overline{u_{\frac{n}{2}-2}}(r)\geq\overline{u_{\frac{n}{2}-2}}(0)+\frac{c_{0}}{2n}r^{2}, \,\,\,\,\,\, \forall \,\, r\geq0.
\end{equation}
Hence, there exists $r_{1} > 0$ such that
\begin{equation}\label{2-10}
  \overline{u_{\frac{n}{2}-2}}(r_{1})>0.
\end{equation}
Next, take a point $x^{2}$ with $|x^{2}-x^{1}|=r_{1}$ as the new center, and make average of $\bar{f}$ at the new center $x^{2}$, i.e.,
\begin{equation}\label{2-11}
\overline{\overline{f}}(r)=\overline{\overline{f}}\big(|x-x^{2}|\big):=\frac{1}{|\partial B_{r}(x^{2})|}\int_{\partial B_{r}(x^{2})}\bar f(x)d\sigma.
\end{equation}
One can easily verify that
\begin{equation}\label{2-12}
\overline{\overline{u_{\frac{n}{2}-2}}}(0)=\overline{u _{\frac{n}{2} - 2}}(x^{2})=:c_{1}>0.
\end{equation}
Then, from \eqref{2-5} and Jensen's inequality, we deduce that $(\overline{\overline{u}},\overline{\overline{u_{1}}},\cdots,\overline{\overline{u_{\frac{n}{2}-1}}})$ satisfies
\begin{equation}\label{2-13}
\left\{{\begin{array}{l} {-\Delta\overline{\overline{u_{\frac{n}{2}-1}}}(r)\geq\big(r+|x^{2}-x^{1}|+|x^{1}|\big)^{-a}\overline{\overline{u}}^{p}(r)\geq0}, \\
{} \\
{-\Delta\overline{\overline{u_{\frac{n}{2}-2}}}(r)=\overline{\overline{u_{\frac{n}{2}-1}}}(r)}, \\ \cdots\cdots \\ {-\Delta\overline{\overline{u}}(r)=\overline{\overline{u _1}}(r)} \\ \end{array}}\right.
\end{equation}
for any $r\geq0$. Using the same method as obtaining the estimate \eqref{2-9}, we conclude that
\begin{equation}\label{2-14}
  \overline{\overline{u_{\frac{n}{2}-2}}}(r)\geq\overline{\overline{u_{\frac{n}{2}-2}}}(0)+\frac{{c_{0}}}{{2n}}r^{2}, \,\,\,\,\,\, \forall \,\, r\geq0.
\end{equation}
Thus we infer from \eqref{2-7}, \eqref{2-12}, \eqref{2-13} and \eqref{2-14} that
\begin{equation}\label{2-15}
\overline{\overline{u_{\frac{n}{2}-1}}}(r)\leq\overline{\overline{u_{\frac{n}{2}-1}}}(0)<0, \,\,\,\,\,\,\,\,\, \overline{\overline{u_{\frac{n}{2}-2}}}(r)\geq\overline{\overline{u_{\frac{n}{2}-2}}}(0)>0, \,\,\,\,\,\, \forall \,\, r\geq0.
\end{equation}
From the third equation in \eqref{2-12} and integrating, we infer that
\begin{equation}\label{2-16}
  \overline{\overline{u_{\frac{n}{2}-3}}}\,'(r)\leq-\frac{c_{1}}{n}r \,\,\,\,\,\, \text{and} \,\,\,\,\,\, \overline{\overline{u_{\frac{n}{2}-3}}}(r)\leq\overline{\overline{u_{\frac{n}{2}-3}}}(0)-\frac{{c_{1}}}{{2n}}r^{2}, \,\,\,\,\,\, \forall \,\, r\geq0.
\end{equation}
Hence, there exists $r_{2}>0$ such that
\begin{equation}\label{2-17}
  \overline{\overline{u_{\frac{n}{2}-3}}}(r_{2})<0.
\end{equation}
Next, we take a point $x^{3}$ with $|x^{3}-x^{2}|=r_{2}$ as the new center and make average of $\bar{\bar{f}}$ at the new center $x^{3}$, i.e.,
\begin{equation}\label{2-18}
\overline{\overline{\overline{f}}}(r)=\overline{\overline{\overline{f}}}\big(|x-x^{3}|\big):=\frac{1}{|\partial B_{r}(x^{3})|}\int_{\partial B_{r}(x^{3})}\overline{\overline{f}}(x)d\sigma.
\end{equation}
It follows that
\begin{equation}\label{2-19}
\overline{\overline{\overline{u_{\frac{n}{2}-3}}}}(0)=\overline{\overline{u _{\frac{n}{2}-3}}}(x^{3})=:-c_{2}<0.
\end{equation}
One can easily verify that $\overline{\overline{\overline{u}}}$ and $\overline{\overline{\overline{u_{i}}}}$ ($i=1,\cdots,\frac{n}{2}-1$) satisfy entirely similar equations as $(\overline{\overline{u}},\overline{\overline{u_{1}}},\cdots,\overline{\overline{u_{\frac{n}{2}-1}}})$ (see \eqref{2-13}). Using the same method as deriving \eqref{2-15}, we arrive at
\begin{equation}\label{2-20}
\overline{\overline{\overline{u_{\frac{n}{2}-1}}}}(r)\leq\overline{\overline{\overline{u_{\frac{n}{2}-1}}}}(0)<0, \,\,\,\,\,\,\, \overline{\overline{\overline{u_{\frac{n}{2}-2}}}}(r)\geq\overline{\overline{\overline{u_{\frac{n}{2}-2}}}}(0)>0, \,\,\,\,\,\,\,
\overline{\overline{\overline{u_{\frac{n}{2}-3}}}}(r)\leq\overline{\overline{\overline{u_{\frac{n}{2}-3}}}}(0)<0
\end{equation}
for any $r\geq0$. Continuing this way, after $\frac{n}{2}$ steps of re-centers (denotes the centers by $x^{1},x^{2},\cdots,x^{\frac{n}{2}}$ and the resulting function coming from taking $\frac{n}{2}$ times averages by $\widetilde{u}$ and $\widetilde{u_{i}}$ for $i=1,2,\cdots,\frac{n}{2}-1$), we finally obtain that
\begin{equation}\label{2-21}
-\Delta\widetilde{u_{\frac{n}{2}-1}}(r)\geq\Big(r+|x^{\frac{n}{2}}-x^{\frac{n}{2}-1}|+\cdots+|x^{2}-x^{1}|+|x^{1}|\Big)^{-a}\widetilde{u}^{p}(r)\geq0,
\end{equation}
and for every $i=1,\cdots,\frac{n}{2}-1$,
\begin{equation}\label{2-22}
(-1)^{i}\widetilde{u_{\frac{n}{2}-i}}(r)\geq(-1)^{i}\widetilde{u_{\frac{n}{2}-i}}(0)>0, \,\,\, \,\,\, (-1)^{\frac{n}{2}}\widetilde{u}(r)\geq(-1)^{\frac{n}{2}}\widetilde{u}(0)>0, \,\,\,\,\,\, \forall \,\, r\geq0.
\end{equation}

Now, if $\frac{n}{2}$ is odd, estimate \eqref{2-22} implies immediately that
\begin{equation}\label{2-23}
  \widetilde{u}(r)\leq\widetilde{u}(0)<0,
\end{equation}
which contradicts the fact that $u\geq0$. Therefore, we only need to deal with the cases that $\frac{n}{2}$ is an even integer hereafter.

For arbitrary $\lambda>0$, define the re-scaling of $u$ by
\begin{equation}\label{2-24}
  u_{\lambda}(x):=\lambda^{\frac{n-a}{p-1}}u(\lambda x).
\end{equation}
Then one can easily verify that equation \eqref{PDE} is invariant under this re-scaling. After $\frac{n}{2}$ steps of re-centers for $u_{\lambda}$, we denote the centers for $u_{\lambda}$ by $x_{\lambda}^{1},x_{\lambda}^{2},\cdots,x_{\lambda}^{\frac{n}{2}}$ and the resulting function coming from taking $\frac{n}{2}$ times averages by $\widetilde{u_{\lambda}}$ and $\widetilde{u_{\lambda,i}}$ for $i=1,2,\cdots,\frac{n}{2}-1$. Then \eqref{2-21} and \eqref{2-22} still hold for $(\widetilde{u_{\lambda}},\widetilde{u_{\lambda,1}},\cdots,\widetilde{u_{\lambda,\frac{n}{2}-1}})$ and $x_{\lambda}^{k}=\frac{1}{\lambda}x_{k}$ for $k=1,\cdots,\frac{n}{2}$, thus one has the following estimate
\begin{equation}\label{2-25}
  |x_{\lambda}^{\frac{n}{2}}-x_{\lambda}^{\frac{n}{2}-1}|+\cdots+|x_{\lambda}^{2}-x_{\lambda}^{1}|+|x_{\lambda}^{1}|\leq
|x^{\frac{n}{2}}-x^{\frac{n}{2}-1}|+\cdots+|x^{2}-x^{1}|+|x^{1}|=:M
\end{equation}
holds uniformly for every $\lambda\geq1$.

Since we have \eqref{2-22} and $\frac{n}{2}$ is even, it follows that
\begin{equation}\label{2-26}
  \widetilde{u}(r)\geq\widetilde{u}(0)>0, \,\,\,\,\,\, \forall \,\, r\geq0,
\end{equation}
and hence
\begin{equation}\label{2-27}
  \widetilde{u_{\lambda}}(r)\geq\widetilde{u_{\lambda}}(0)=\lambda^{\frac{n-a}{p-1}}\widetilde{u}(0)>0, \,\,\,\,\,\, \forall \,\, r\geq0.
\end{equation}
It is clear that one can choose $\lambda$ sufficiently large, such that $\widetilde{u_{\lambda}}(0)$ be as large as we wish. Therefore, by the estimates \eqref{2-25} and \eqref{2-27}, we may assume that, for any given $l_{0}>0$ (to be determined later), we already have
\begin{equation}\label{2-28}
\widetilde{u}(r)\geq l_{0}\geq l_{0}r^{\alpha_{0}}, \,\,\,\,\,\,\,\,\, \forall \,\, 0\leq r\leq1,
\end{equation}
where
\begin{equation}\label{2-29}
  \alpha_{0}:=\max\Big\{1,\frac{2n}{p}\Big\}\geq1,
\end{equation}
or else we may replace $u$ by $u_{\lambda}$ with $\lambda$ large enough. As a consequence, we infer from \eqref{2-21}, \eqref{2-25} and \eqref{2-28} that such solution $u$ satisfies
\begin{align}\label{2-30}
-\Delta\widetilde{u_{\frac{n}{2}-1}}(r)&\geq\Big(r+|x^{\frac{n}{2}}-x^{\frac{n}{2}-1}|+\cdots+|x^{2}-x^{1}|+|x^{1}|\Big)^{-a}\widetilde{u}^{p}(r) \nonumber \\
&\geq\big(1+M\big)^{-a}\,l_{0}^{p}\,r^{\alpha_{0}p}\\
&=:C_{0}\,l_{0}^{p}\,r^{\alpha_{0}p}, \,\,\,\,\,\,\,\,\,\,\,\,  \forall \,\, 0\leq r\leq1, \nonumber
\end{align}
where $C_{0}:=(1+M)^{-a}\in(0,1)$. Integrating both sides of \eqref{2-30} from $0$ to $r$ twice and taking into account of \eqref{2-22} yield
\begin{equation}\label{2-31}
  \widetilde{u_{\frac{n}{2}-1}}(r)<-\frac{C_{0}l_{0}^{p}}{(\alpha_{0}p+n)(\alpha_{0}p+2)}r^{\alpha_{0}p+2}, \,\,\,\,\,\, \forall \,\, 0\leq r\leq1.
\end{equation}
This implies
\begin{equation}\label{2-32}
  -\frac{1}{r^{n-1}}\left(r^{n-1}\widetilde{u_{\frac{n}{2}-2}}\,'(r)\right)'<-\frac{C_{0}l_{0}^{p}}{(\alpha_{0}p+n)(\alpha_{0}p+2)}r^{\alpha_{0}p+2},
\end{equation}
and consequently,
\begin{equation}\label{2-33}
  \widetilde{u_{\frac{n}{2}-2}}(r)>\frac{C_{0}l_{0}^{p}}{(\alpha_{0}p+n)(\alpha_{0}p+2)(\alpha_{0}p+n+2)(\alpha_{0}p+4)}r^{\alpha_{0}p+4}, \,\,\,\,\,\, \forall \,\, 0\leq r\leq1.
\end{equation}
Continuing this way, since $\frac{n}{2}$ is an even integer, by iteration, we can finally arrive at
\begin{equation}\label{2-34}
  \widetilde{u}(r)>\frac{C_{0}l_{0}^{p}}{(\alpha_{0}p+2n)^{n}}r^{\alpha_{0}p+n}, \,\,\,\,\,\, \forall \,\, 0\leq r\leq1.
\end{equation}
Now, define
\begin{equation}\label{2-35}
  \alpha_{k+1}:=2\alpha_{k}p\geq\alpha_{k}p+2n \,\,\,\,\,\, \text{and} \,\,\,\,\,\, l_{k+1}:=\frac{C_{0}l_{k}^{p}}{(2\alpha_{k}p)^{n}}
\end{equation}
for $k=0,1,\cdots$. Then \eqref{2-34} implies
\begin{equation}\label{2-36}
\widetilde{u}(r)>\frac{C_{0}l_{0}^{p}}{(2\alpha_{0}p)^{n}}r^{2\alpha_{0}p}=l_{1}r^{\alpha_{1}}, \,\,\,\,\, \forall \,\, r\in[0,1].
\end{equation}
Suppose we have $\widetilde{u}(r)\geq l_{k}r^{\alpha_{k}}$, then go through the entire process as above, we can derive $\widetilde{u}(r)\geq l_{k+1}r^{\alpha_{k+1}}$ for any $0\leq r\leq1$. Therefore, one can prove by induction that
\begin{equation}\label{2-37}
  \widetilde{u}(r)\geq l_{k}r^{\alpha_{k}}, \,\,\,\,\,\, \forall \,\, r\in[0,1], \,\,\,\,\,\, \forall \,\, k\in\mathbb{N}.
\end{equation}
Through direct calculations, we have
\begin{eqnarray}\label{2-38}
  l_{k}&=&\frac{C_{0}^{\frac{p^{k}-1}{p-1}}l_{0}^{p^{k}}}{(2p)^{n(k+(k-1)p+(k-2)p^{2}+\cdots+p^{k-1})}\alpha_{0}^{\frac{n(p^{k}-1)}{p-1}}} \\
 \nonumber &=& \frac{C_{0}^{\frac{p^{k}-1}{p-1}}l_{0}^{p^{k}}(2p)^{\frac{nk}{p-1}}}{(2p)^{\frac{n(p^{k+1}-p)}{(p-1)^{2}}}\alpha_{0}^{\frac{n(p^{k}-1)}{p-1}}}
\geq (2p)^{\frac{nk}{p-1}}\left(\frac{C_{0}^{\frac{1}{p-1}}l_{0}}{(2p)^{\frac{np}{(p-1)^{2}}}\alpha_{0}^{\frac{n}{p-1}}}\right)^{p^{k}}
\end{eqnarray}
for $k=0,1,2,\cdots$. Now we take
\begin{equation}\label{2-39}
  l_{0}:=2C_{0}^{\frac{-1}{p-1}}(2p)^{\frac{np}{(p-1)^{2}}}\alpha_{0}^{\frac{n}{p-1}},
\end{equation}
then from \eqref{2-37}, \eqref{2-38} and \eqref{2-39}, we deduce that
\begin{equation}\label{2-40}
  \widetilde{u}(1)\geq(2p)^{\frac{nk}{p-1}}2^{p^{k}}\rightarrow+\infty, \,\,\,\,\,\, \text{as} \,\, k\rightarrow\infty.
\end{equation}
This is absurd. Therefore, \eqref{2-1} must hold, that is, $u_{\frac{n}{2}-1}=(-\Delta)^{\frac{n}{2}- 1}u\geq0$.

\textbf{\emph{Step 2.}} Next, we will show that all the other $u_{i}$ ($i=1,\cdots,\frac{n}{2}-2$) must be nonnegative, that is,
\begin{equation}\label{2-41}
u_{\frac{n}{2}-i}(x)\geq0, \,\,\,\,\,\,\,\,\,\,\, \forall \,\, i=2,3,\cdots,\frac{n}{2}-1, \,\,\,\,\,\, \forall \,\, x\in\mathbb{R}^{n}.
\end{equation}
Suppose on the contrary that, there exists some $2\leq i\leq\frac{n}{2}-1$ and some $x^{0}\in\mathbb{R}^{n}$ such that
\begin{equation}\label{2-42}
  u_{\frac{n}{2}-1}(x)\geq0, \,\,\,\,\, u_{\frac{n}{2}-2}(x)\geq0, \,\,\,\, \cdots, \,\,\,\, u_{\frac{n}{2}-i+1}(x)\geq0, \,\,\,\,\,\, \forall \,\, x\in\mathbb{R}^{n},
\end{equation}
\begin{equation}\label{2-43}
  u_{\frac{n}{2}-i}(x^{0})<0.
\end{equation}
Then, repeating the similar ``re-centers and iteration" arguments as in Step 1, after $\frac{n}{2}-i+1$ steps of re-centers (denotes the centers by $\bar{x}^{1},\bar{x}^{2},\cdots,\bar{x}^{\frac{n}{2}-i+1}$), the signs of the resulting functions $\widetilde{u_{\frac{n}{2}-j}}$ ($j=i,\cdots,\frac{n}{2}-1$) and $\widetilde{u}$ satisfy
\begin{equation}\label{2-44}
  (-1)^{j-i+1}\widetilde{u_{\frac{n}{2}-j}}(r)\geq(-1)^{j-i+1}\widetilde{u_{\frac{n}{2}-j}}(0)>0, \,\,\,\,\,\,
(-1)^{\frac{n}{2}-i+1}\widetilde{u}(r)\geq(-1)^{\frac{n}{2}-i+1}\widetilde{u}(0)>0
\end{equation}
for any $r\geq0$. Since $u\geq0$, it follows immediately from \eqref{2-44} that $\frac{n}{2}-i+1$ is even and
\begin{equation}\label{2-45}
  \widetilde{u}(r)\geq\widetilde{u}(0)>0, \,\,\,\,\,\, \forall \,\, r\geq0.
\end{equation}
Furthermore, since $\frac{n}{2}-i$ is odd, we infer from \eqref{2-44} that
\begin{equation}\label{2-48}
  -\Delta\widetilde{u}(r)=\widetilde{u_{1}}(r)\leq\widetilde{u_{1}}(0)=:-\widetilde{c}<0, \,\,\,\,\,\, \forall \,\, r\geq0,
\end{equation}
and hence, by integrating, one has
\begin{equation}\label{2-49}
  \widetilde{u}(r)\geq\widetilde{u}(0)+\frac{\widetilde{c}}{2n}r^{2}>\frac{\widetilde{c}}{2n}r^{2}, \,\,\,\,\,\, \forall \,\, r\geq0.
\end{equation}
Therefore, if we assume that $u(x)=o(|x|^{2})$ as $|x|\rightarrow+\infty$, we will get a contradiction from \eqref{2-49}.

Or, if we assume that $0\leq a\leq2+2p$, combining \eqref{2-49} with the estimate \eqref{2-21}, we get
\begin{eqnarray}\label{2-46}
  -\Delta\widetilde{u_{\frac{n}{2}-1}}(r)&\geq&\left(r+|\bar{x}^{\frac{n}{2}-i+1}-\bar{x}^{\frac{n}{2}-i}|+\cdots+|\bar{x}^{2}-\bar{x}^{1}|+|\bar{x}^{1}|
\right)^{-a}\widetilde{u}^{p}(r) \\
  \nonumber &\geq&\left(\frac{\widetilde{c}}{2n}\right)^{p}r^{2p-a}
\end{eqnarray}
for $r\geq r_{0}$ sufficiently large. Now, by a direct integration on \eqref{2-46}, we get, if $0\leq a<2+2p$, then
\begin{equation}\label{2-47}
\widetilde{u_{\frac{n}{2}-1}}(r)\leq\widetilde{u_{\frac{n}{2}-1}}(r_{0})-\left(\frac{\widetilde{c}}{2n}\right)^{p}
\frac{r^{2+2p-a}-r_{0}^{2+2p-a}}{(n+2p-a)(2+2p-a)}\rightarrow-\infty, \,\,\,\,\,\, \text{as} \,\,\, r\rightarrow\infty;
\end{equation}
if $a=2+2p$, then
\begin{equation}\label{2-47'}
\widetilde{u_{\frac{n}{2}-1}}(r)\leq\widetilde{u_{\frac{n}{2}-1}}(r_{0})-\left(\frac{\widetilde{c}}{2n}\right)^{p}\frac{\ln r-\ln r_{0}}{n-2}\rightarrow-\infty, \,\,\,\,\,\, \text{as} \,\,\, r\rightarrow\infty.
\end{equation}
This contradicts $u_{\frac{n}{2}-1}\geq0$ and thus \eqref{2-41} must hold. This concludes the proof of Lemma \ref{lemma0}.
\end{proof}

In the following, we will continue carrying out our proof under the same assumptions as Lemma \ref{lemma0}.

By Lemma \ref{lemma0}, we can deduce from $-\Delta u\geq0$, $u\geq0$, $u(\bar{x})>0$ and maximum principle that
\begin{equation}\label{2-50}
  u(x)>0, \,\,\,\,\,\,\, \forall \,\, x\in\mathbb{R}^{n}.
\end{equation}
Then, by maximum principle, Lemma 2.1 from Chen and Lin \cite{CLin} and induction, we can also infer further from $(-\Delta)^{i} u\geq0$ ($i=1,\cdots,\frac{n}{2}-1$), $u>0$ and equation \eqref{PDE} that
\begin{equation}\label{2-51}
  (-\Delta)^{i}u(x)>0, \,\,\,\,\,\,\,\, \forall \,\, i=1,\cdots,\frac{n}{2}-1, \,\,\,\, \forall \,\, x\in\mathbb{R}^{n}.
\end{equation}

Next, we will try to obtain contradictions by discussing two different cases $0\leq a<2$ and $a\geq2$ separately.

\emph{Case i)} $0\leq a<2$. We will also need the following lemma concerning the removable singularity.
\begin{lem}\label{lemma1}
Suppose $u$ is harmonic in $B_{R}(0)\setminus\{0\}$ and satisfies
\begin{equation*}
  u(x)=\left\{
         \begin{array}{ll}
           o(\ln|x|), \,\,\,\,\,\,  n=2,  \\
           \\
           o(|x|^{2-n}), \,\,\,\,\,\, n\geq3,
         \end{array}
       \right. \,\,\,\,\,\,\, \text{as} \,\,\, |x|\rightarrow0.
\end{equation*}
Then $u$ can be defined at $0$ so that it is $C^{2}$ and harmonic in $B_{R}(0)$.
\end{lem}
Lemma \ref{lemma1} can be proved directly by using the Poisson integral formula and maximum principles, so we omit the details.

Now we will first show that $(-\Delta)^{\frac{n}{2}-1}u$ satisfies the following integral equation
\begin{equation}\label{2c1}
  (-\Delta)^{\frac{n}{2}-1}u(x)=\int_{\mathbb{R}^{n}}\frac{R_{2,n}}{|x-y|^{n-2}}\frac{u^{p}(y)}{|y|^{a}}dy, \,\,\,\,\,\,\,\,\,\, \forall \,\, x\in\mathbb{R}^{n},
\end{equation}
where the Riesz potential's constants $R_{\alpha,n}:=\frac{\Gamma(\frac{n-\alpha}{2})}{\pi^{\frac{n}{2}}2^{\alpha}\Gamma(\frac{\alpha}{2})}$ for $0<\alpha<n$.

To this end, for arbitrary $R>0$, let $f_{1}(u)(x):=\frac{u^{p}(x)}{|x|^{a}}$ and
\begin{equation}\label{2c2}
v_{1}^{R}(x):=\int_{B_R(0)}G_R(x,y)f_{1}(u)(y)dy,
\end{equation}
where the Green's function for $-\Delta$ on $B_R(0)$ is given by
\begin{equation}\label{Green}
  G_R(x,y)=R_{2,n}\bigg[\frac{1}{|x-y|^{n-2}}-\frac{1}{\big(|x|\cdot\big|\frac{Rx}{|x|^{2}}-\frac{y}{R}\big|\big)^{n-2}}\bigg], \,\,\,\, \text{if} \,\, x,y\in B_{R}(0),
\end{equation}
and $G^{2}_{R}(x,y)=0$ if $x$ or $y\in\mathbb{R}^{n}\setminus B_{R}(0)$.

Then, since $0\leq a<2$, we can derive that $v_{1}^{R}\in C^{2}(\mathbb{R}^{n}\setminus\{0\})\cap C(\mathbb{R}^{n})$ and satisfies
\begin{equation}\label{2c3}\\\begin{cases}
-\Delta v_{1}^{R}(x)=\frac{u^{p}(x)}{|x|^{a}},\ \ x\in B_R(0)\setminus\{0\},\\
v_{1}^{R}(x)=0,\ \ \ \ \ \ \ x\in \mathbb{R}^{n}\setminus B_R(0).
\end{cases}\end{equation}
Let $w_{1}^R(x):=(-\Delta)^{\frac{n}{2}-1}u(x)-v_{1}^R(x)\in C^{2}(\mathbb{R}^{n}\setminus\{0\})\cap C(\mathbb{R}^{n})$. By Lemma \ref{lemma0}, Lemma \ref{lemma1}, \eqref{PDE} and \eqref{2c3}, we have $w_{1}^R\in C^{2}(\mathbb{R}^{n})$ and satisfies
\begin{equation}\label{2c4}\\\begin{cases}
-\Delta w_{1}^R(x)=0,\ \ \ \ x\in B_R(0),\\
w_{1}^{R}(x)>0, \,\,\,\,\, x\in \mathbb{R}^{n}\setminus B_R(0).
\end{cases}\end{equation}
By maximum principle, we deduce that for any $R>0$,
\begin{equation}\label{2c5}
  w_{1}^R(x)=(-\Delta)^{\frac{n}{2}-1}u(x)-v_{1}^{R}(x)>0, \,\,\,\,\,\,\, \forall \,\, x\in\mathbb{R}^{n}.
\end{equation}
Now, for each fixed $x\in\mathbb{R}^{n}$, letting $R\rightarrow\infty$ in \eqref{2c5}, we have
\begin{equation}\label{2c6}
(-\Delta)^{\frac{n}{2}-1}u(x)\geq\int_{\mathbb{R}^{n}}\frac{R_{2,n}}{|x-y|^{n-2}}f_{1}(u)(y)dy=:v_{1}(x)>0.
\end{equation}
Take $x=0$ in \eqref{2c6}, we get
\begin{equation}\label{2c7}
  \int_{\mathbb{R}^{n}}\frac{u^{p}(y)}{|y|^{n-2+a}}dy<+\infty.
\end{equation}
One can easily observe that $v_{1}\in C^{2}(\mathbb{R}^{n}\setminus\{0\})\cap C(\mathbb{R}^{n})$ is a solution of
\begin{equation}\label{2c8}
-\Delta v_{1}(x)=\frac{u^{p}(x)}{|x|^{a}},  \,\,\,\,\,\,\, x\in \mathbb{R}^n\setminus\{0\}.
\end{equation}
Define $w_{1}(x):=(-\Delta)^{\frac{n}{2}-1}u(x)-v_{1}(x)\in C^{2}(\mathbb{R}^{n}\setminus\{0\})\cap C(\mathbb{R}^{n})$. Then, by Lemma \ref{lemma1}, \eqref{PDE}, \eqref{2c6} and \eqref{2c8}, we have $w_{1}\in C^{2}(\mathbb{R}^{n})$ and satisfies
\begin{equation}\label{2c9}\\\begin{cases}
-\Delta w_{1}(x)=0, \,\,\,\,\,  x\in \mathbb{R}^n,\\
w_{1}(x)\geq0 \,\,\,\,\,\,  x\in \mathbb{R}^n.
\end{cases}\end{equation}
From Liouville theorem for harmonic functions, we can deduce that
\begin{equation}\label{2c10}
   w_{1}(x)=(-\Delta)^{\frac{n}{2}-1}u(x)-v_{1}(x)\equiv C_{1}\geq0.
\end{equation}
Therefore, we have
\begin{equation}\label{2c11}
  (-\Delta)^{\frac{n}{2}-1}u(x)=\int_{\mathbb{R}^{n}}\frac{R_{2,n}}{|x-y|^{n-2}}\frac{u^{p}(y)}{|y|^{a}}dy+C_{1}=:f_{2}(u)(x)>C_{1}\geq0.
\end{equation}

Next, for arbitrary $R>0$, let
\begin{equation}\label{2c12}
v_{2}^R(x):=\int_{B_R(0)}G_R(x,y)f_{2}(u)(y)dy.
\end{equation}
Then, we can get
\begin{equation}\label{2c13}\\\begin{cases}
-\Delta v_2^R(x)=f_{2}(u)(x),\ \ x\in B_R(0),\\
v_2^R(x)=0,\ \ \ \ \ \ \ x\in \mathbb{R}^{n}\setminus B_R(0).
\end{cases}\end{equation}
Let $w_2^R(x):=(-\Delta)^{\frac{n}{2}-2}u(x)-v_2^R(x)$. By Lemma \ref{lemma0}, \eqref{2c11} and \eqref{2c13}, we have
\begin{equation}\label{2c14}\\\begin{cases}
-\Delta w_2^R(x)=0,\ \ \ \ x\in B_R(0),\\
w_2^R(x)>0, \,\,\,\,\, x\in \mathbb{R}^{n}\setminus B_R(0).
\end{cases}\end{equation}
By maximum principle, we deduce that for any $R>0$,
\begin{equation}\label{2c15}
  w_2^R(x)=(-\Delta)^{\frac{n}{2}-2}u(x)-v_2^{R}(x)>0, \,\,\,\,\,\,\, \forall \,\, x\in\mathbb{R}^{n}.
\end{equation}
Now, for each fixed $x\in\mathbb{R}^{n}$, letting $R\rightarrow\infty$ in \eqref{2c15}, we have
\begin{equation}\label{2c16}
(-\Delta)^{\frac{n}{2}-2}u(x)\geq\int_{\mathbb{R}^{n}}\frac{R_{2,n}}{|x-y|^{n-2}}f_{2}(u)(y)dy=:v_{2}(x)>0.
\end{equation}
Take $x=0$ in \eqref{2c16}, we get
\begin{equation}\label{2c17}
  \int_{\mathbb{R}^{n}}\frac{C_{1}}{|y|^{n-2}}dy\leq\int_{\mathbb{R}^{n}}\frac{f_{2}(u)(y)}{|y|^{n-2}}dy<+\infty,
\end{equation}
it follows easily that $C_{1}=0$, and hence we have proved \eqref{2c1}, that is,
\begin{equation}\label{2c18}
  (-\Delta)^{\frac{n}{2}-1}u(x)=f_{2}(u)(x)=\int_{\mathbb{R}^{n}}\frac{R_{2,n}}{|x-y|^{n-2}}\frac{u^{p}(y)}{|y|^{a}}dy.
\end{equation}
One can easily observe that $v_{2}$ is a solution of
\begin{equation}\label{2c19}
-\Delta v_{2}(x)=f_{2}(u)(x),  \,\,\,\,\, x\in \mathbb{R}^n.
\end{equation}
Define $w_{2}(x):=(-\Delta)^{\frac{n}{2}-2}u(x)-v_{2}(x)$, then it satisfies
\begin{equation}\label{2c20}\\\begin{cases}
-\Delta w_{2}(x)=0, \,\,\,\,\,  x\in \mathbb{R}^n,\\
w_{2}(x)\geq0 \,\,\,\,\,\,  x\in\mathbb{R}^n.
\end{cases}\end{equation}
From Liouville theorem for harmonic functions, we can deduce that
\begin{equation}\label{2c21}
   w_{2}(x)=(-\Delta)^{\frac{n}{2}-2}u(x)-v_{2}(x)\equiv C_{2}\geq0.
\end{equation}
Therefore, we have proved that
\begin{equation}\label{2c22}
  (-\Delta)^{\frac{n}{2}-2}u(x)=\int_{\mathbb{R}^{n}}\frac{R_{2,n}}{|x-y|^{n-2}}f_{2}(u)(y)dy+C_{2}=:f_{3}(u)(x)>C_{2}\geq0.
\end{equation}
Through the same methods as above, we can prove that $C_{2}=0$, and hence
\begin{equation}\label{2c23}
  (-\Delta)^{\frac{n}{2}-2}u(x)=f_{3}(u)(x)=\int_{\mathbb{R}^{n}}\frac{R_{2,n}}{|x-y|^{n-2}}f_{2}(u)(y)dy.
\end{equation}
Continuing this way, defining
\begin{equation}\label{2c24}
  f_{k+1}(u)(x):=\int_{\mathbb{R}^{n}}\frac{R_{2,n}}{|x-y|^{n-2}}f_{k}(u)(y)dy
\end{equation}
for $k=1,2,\cdots,\frac{n}{2}$, then by Lemma \ref{lemma0} and induction, we have
\begin{equation}\label{2c25}
  (-\Delta)^{\frac{n}{2}-k}u(x)=f_{k+1}(u)(x)=\int_{\mathbb{R}^{n}}\frac{R_{2,n}}{|x-y|^{n-2}}f_{k}(u)(y)dy
\end{equation}
for $k=1,2,\cdots,\frac{n}{2}-1$, and
\begin{equation}\label{2c50}
  u(x)\geq f_{\frac{n}{2}+1}(u)(x)=\int_{\mathbb{R}^{n}}\frac{R_{2,n}}{|x-y|^{n-2}}f_{\frac{n}{2}}(u)(y)dy.
\end{equation}
In particular, it follows from \eqref{2c25} and \eqref{2c50} that
\begin{eqnarray}\label{2c51}
  && +\infty>(-\Delta)^{\frac{n}{2}-k}u(0)=\int_{\mathbb{R}^{n}}\frac{R_{2,n}}{|y|^{n-2}}f_{k}(u)(y)dy \\
 \nonumber &\geq& \int_{\mathbb{R}^{n}}\frac{R_{2,n}}{|y^{k}|^{n-2}}\int_{\mathbb{R}^{n}}\frac{R_{2,n}}{|y^{k}-y^{k-1}|^{n-2}}\cdots
  \int_{\mathbb{R}^{n}}\frac{R_{2,n}}{|y^{2}-y^{1}|^{n-2}}\frac{u^{p}(y^{1})}{|y^{1}|^{a}}dy^{1}dy^{2}\cdots dy^{k}
\end{eqnarray}
for $k=1,2,\cdots,\frac{n}{2}-1$, and
\begin{eqnarray}\label{formula}
  && +\infty>u(0)\geq\int_{\mathbb{R}^{n}}\frac{R_{2,n}}{|y|^{n-2}}f_{\frac{n}{2}}(u)(y)dy \\
 \nonumber &\geq& \int_{\mathbb{R}^{n}}\frac{R_{2,n}}{|y^{\frac{n}{2}}|^{n-2}}\left(\int_{\mathbb{R}^{n}}\frac{R_{2,n}}{|y^{\frac{n}{2}}-y^{\frac{n}{2}-1}|^{n-2}}\cdots
  \int_{\mathbb{R}^{n}}\frac{R_{2,n}}{|y^{2}-y^{1}|^{n-2}}\frac{u^{p}(y^{1})}{|y^{1}|^{a}}dy^{1}dy^{2}\cdots dy^{\frac{n}{2}-1}\right)dy^{\frac{n}{2}}.
\end{eqnarray}
From the properties of Riesz potential, we have the following formula (see \cite{Stein}), that is, for any $\alpha_{1},\alpha_{2}\in(0,n)$ such that $\alpha_{1}+\alpha_{2}\in(0,n)$, one has
\begin{equation}\label{2c26}
  \int_{\mathbb{R}^{n}}\frac{R_{\alpha_{1},n}}{|x-y|^{n-\alpha_{1}}}\cdot\frac{R_{\alpha_{2},n}}{|y-z|^{n-\alpha_{2}}}dy
=\frac{R_{\alpha_{1}+\alpha_{2},n}}{|x-z|^{n-(\alpha_{1}+\alpha_{2})}}.
\end{equation}
Now, by applying the formula \eqref{2c26} and direct calculations, we obtain that
\begin{eqnarray}\label{2c27}
  && \int_{\mathbb{R}^{n}}\frac{R_{2,n}}{|y^{\frac{n}{2}}-y^{\frac{n}{2}-1}|^{n-2}}\cdots
\int_{\mathbb{R}^{n}}\frac{R_{2,n}}{|y^{3}-y^{2}|^{n-2}}\cdot\frac{R_{2,n}}{|y^{2}-y^{1}|^{n-2}}dy^{2}\cdots dy^{\frac{n}{2}-1} \\
 \nonumber &=& \frac{R_{n-2,n}}{|y^{\frac{n}{2}}-y^{1}|^{2}}.
\end{eqnarray}

Now, we can deduce from \eqref{formula}, \eqref{2c27} and Fubini's theorem that
\begin{eqnarray}\label{contradiction}
  +\infty&>&u(0)\geq\int_{\mathbb{R}^{n}}\frac{R_{2,n}}{|y^{\frac{n}{2}}|^{n-2}}\left(\int_{\mathbb{R}^{n}}\frac{R_{n-2,n}}{|y^{\frac{n}{2}}-y^{1}|^{2}}\cdot
  \frac{u^{p}(y^{1})}{|y^{1}|^{a}}dy^{1}\right)dy^{\frac{n}{2}} \\
 \nonumber &=& \frac{1}{(2\pi)^{n}}\int_{\mathbb{R}^{n}}\frac{1}{|y|^{n-2}}\left(\int_{\mathbb{R}^{n}}\frac{1}{|y-z|^{2}}\cdot\frac{u^{p}(z)}{|z|^{a}}dz\right)dy.
\end{eqnarray}

We will get a contradiction from \eqref{contradiction}. Indeed, if we assume that $u$ is not identically zero, then by the integrability \eqref{2c7}, we have
\begin{equation}\label{2c60}
0<C_{0}:=\int_{\mathbb{R}^n}\frac{1}{|z|^{n-2}}\cdot\frac{u^p(z)}{|z|^a}dz<+\infty.
\end{equation}
For any given $|y|\geq 3$, if $|z|\geq\big(\ln|y|\big)^{-\frac{1}{n-2}}$, then one has immediately
\begin{equation}\label{2c61}
|y-z|\leq|y|+|z|\leq\left(|y|\big(\ln|y|\big)^{\frac{1}{n-2}}+1\right)|z|\leq 2|y|\big(\ln|y|\big)^{\frac{1}{n-2}}|z|.
\end{equation}
Thus it follows from \eqref{2c60} and \eqref{2c61} that, there exists a $R_{0}$ sufficiently large, such that, for any $|y|\geq R_{0}$, we have
\begin{eqnarray}\label{2c63}
\int_{\mathbb{R}^{n}}\frac{1}{|y-z|^2}\cdot\frac{u^p(z)}{|z|^a}dz&\geq& \frac{1}{4|y|^2\ln|y|}\int_{|z|\geq\big(\ln|y|\big)^{-\frac{1}{n-2}}}\frac{1}{|z|^{n-2}}\cdot\frac{u^p(z)}{|z|^a}dz \\
\nonumber &\geq&  \frac{C_{0}}{8|y|^2\ln|y|}.
\end{eqnarray}

Therefore, we can finally deduce from \eqref{contradiction} and \eqref{2c63} that
\begin{equation}\label{final}
 +\infty>u(0)\geq\frac{C_{0}}{8(2\pi)^{n}}\int_{|y|\geq R_{0}}\frac{1}{|y|^{n}\ln|y|}dy=+\infty,
\end{equation}
which is a contradiction! Therefore, we must have $u\equiv0$ in $\mathbb{R}^{n}$.

\emph{Case ii)} $a\geq2$. From Lemma \ref{lemma1} and \eqref{PDE}, we derive that $u\in C^{n}(\mathbb{R}^{n}\setminus\{0\})\cap C^{n-2}(\mathbb{R}^{n})$ and $u_{i}=(-\Delta)^{i} u\in C^{n-2i}(\mathbb{R}^{n}\setminus\{0\})\cap C^{n-2-2i}(\mathbb{R}^{n})$ ($i=1,\cdots,\frac{n}{2}-1$) form a positive solution to the following Lane-Emden-Hardy system
\begin{equation}\label{PDES}
\left\{{\begin{array}{l} {-\Delta u_{\frac{n}{2}-1}(x)=\frac{u^{p}(x)}{|x|^{a}}, \,\,\,\,\,\, x\in\mathbb{R}^{n}\setminus\{0\}},\\  {} \\ {-\Delta u_{\frac{n}{2}-2}(x)=u_{\frac{n}{2}-1}(x), \,\,\,\,\,\, x\in\mathbb{R}^{n}}, \\ \cdots\cdots \\ {-\Delta u(x)=u_1(x), \,\,\,\,\,\, x\in\mathbb{R}^{n}}. \\ \end{array}}\right.
\end{equation}
Since $u\in C^{n}(\mathbb{R}^{n}\setminus\{0\})\cap C^{n-2}(\mathbb{R}^{n})$, $u>0$, $u_{i}>0$, $\Delta u<0$ and $\Delta u_{i}<0$ for $|x|>0$, by direct calculations, we can deduce that
\begin{equation}\label{2-3-17}
  \frac{d}{dr}\overline{u}(r)\leq0, \,\,\,\,\,\,\,\,\, \frac{d}{dr}\overline{u_{i}}(r)\leq0 \,\,\,\,\,\,\,\,\,\,\,\, \text{for any} \,\,\, 0<r<\infty.
\end{equation}
By taking the spherical average of equations of \eqref{PDES}  with respect to the center $0$ and Jensen's inequality, we have
\begin{equation}\label{2-3-18}
  \overline{u}''(r)+\frac{n-1}{r}\overline{u}'(r)+\overline{u_{1}}(r)=0, \,\,\,\,\,\,\,\, \overline{u_{i}}''(r)+\frac{n-1}{r}\overline{u_{i}}'(r)+\overline{u_{i+1}}(r)=0,
   \,\,\,\,\,\,\,\,\, \forall \,\, r\geq0,
\end{equation}
\begin{equation}\label{2-3-19}
  \overline{u_{\frac{n}{2}-1}}''(r)+\frac{n-1}{r}\overline{u_{\frac{n}{2}-1}}'(r)+r^{-a}\overline{u}^{p}(r)\leq0, \,\,\,\,\,\,\,\,\, \forall \,\, r>0.
\end{equation}
Thus we can infer from \eqref{2-3-17} and \eqref{2-3-19} that, for any $0<r\leq1$,
\begin{equation}\label{2-3-20}
  -\overline{u_{\frac{n}{2}-1}}'(r)\geq r^{1-n}\int_{0}^{r}s^{n-1-a}\overline{u}^{p}(s)ds\geq c^pr^{1-n}\int_{0}^{r}s^{n-1-a}ds,
\end{equation}
where $c:=\min_{|x|\leq1}u(x)>0$. For $2\leq a<n$, one can deduce further from \eqref{2-3-20} that
\begin{equation}\label{2-3-21}
  -\overline{u_{\frac{n}{2}-1}}'(r)\geq \frac{c^p}{n-a}r^{1-a}, \,\,\,\,\,\,\,\, \forall \,\, 0<r\leq1.
\end{equation}
Integrating both sides of \eqref{2-3-21} from $0$ to $1$ yields that
\begin{eqnarray}\label{2-3-22}
  (-\Delta)^{\frac{n}{2}-1}u(0)=\overline{u_{\frac{n}{2}-1}}(0)&\geq&\overline{u_{\frac{n}{2}-1}}(1)
  +\frac{c^p}{n-a}\int_{0}^{1}s^{1-a}ds \\
  \nonumber &\geq& \frac{c^p}{n-a}\int_{0}^{1}s^{1-a}ds=+\infty,
\end{eqnarray}
which is a contradiction! Therefore, we must have $u\equiv0$ in $\mathbb{R}^{n}$.

This concludes the proof of Theorem \ref{Thm0}.

\end{document}